\documentclass[11pt, oneside]{article}   	

\usepackage{ amssymb }
\usepackage{amsfonts} 
\usepackage{amsmath}                           
\usepackage{amsthm}
\usepackage{tikz-cd}
\usepackage{subfig}
\usepackage{pgfplots}
\usepackage{lipsum}

\title{SUBALGEBRA DEPTH AND DOUBLE CROSSED PRODUCTS}
\author{Alberto Hern\'andez Alvarado \\ Escuela de Matem\'atica - CIMPA \footnote{Centro de Investigaciones en Matem\'atica Pura y Aplicada} \\Universidad de Costa Rica}
\date{}							

\begin{document}
\maketitle

\newtheorem{thm}{Theorem}[section]
\newtheorem{ex}[thm]{Example}
\newtheorem{de}[thm]{Definition}
\newtheorem{lem}[thm]{Lemma}
\newtheorem{pr}[thm]{Proposition}
\newtheorem{co}[thm]{Corollary}
\newtheorem{rem}[thm]{Remark}

\newcommand \rrhu {\mathrel{\mathpalette\lrhup\relax}}
\newcommand \lrhup[2]{\ooalign{$#1\rightharpoonup$\cr $#2 \hspace{1mm}\rightharpoonup$\cr}}

\newcommand \llhu {\mathrel{\mathpalette\rlhup\relax}}
\newcommand \rlhup[2]{\ooalign{$#1\leftharpoonup$\cr $#2 \hspace{1mm}\leftharpoonup$\cr}}

\newcommand \vtick{\textsc{\char13}}
\newcommand{\exedout}{\rule{0.8\textwidth}{0.5\textwidth}}

\begin{abstract}

In this paper we explore the concept of depth of a ring extension when the overall algebra factorises as a product of two subalgebras, in particular the case of finite dimensional Hopf algebras. As a result we generalise the results by Kadison and Young \cite{HKY} on depth of a Hopf algebra $H$ in its smash product with a finite dimensional left $H$-module algebra $A$, $A\#H$ to the context of generalised smash products $Q^{*op}\#_\psi H$ \cite{Bz1} where $Q$ is the quotient module coalgebra associated to the extension $R\subseteq H$ of finite dimensional Hopf algebras \cite{Ka2}\cite{HKY}\cite{H}. Moreover, following the construction of double crossed products in \cite{Ma} and  \cite{Ma1} we use our result on factorisation algebras to get a general result on the depth of the extension of a Hopf algebra $H$ in its Drinfel\vtick d double $D(H)$.
\footnote{$\mathbf{Keywords : }$ Depth, Factorisation Algebra, Smash Product, Drinfel\vtick d Double, Double Crossed Product, Normal Extension.}\footnote{$\mathbf{Subject}$ $\mathbf{classification :}$ $20C05$, $20G05$, $16W30$, $17B37$, $13E10$.}

\end{abstract}

\section{Introduction and preliminaries}

\subsection{Introduction}

The concept of depth \cite{KN},\cite{Ka1},\cite{KK1},\cite{BDK} relates a similarity between tensor powers of the regular representation of a ring extension $B\subseteq A$ with structure properties of the given extension. In a sense, for a finite dimensional Hopf algebra extension $R\subseteq H$ depth gives a measure of how close (or far) from being normal the extension is \cite{KN},\cite{KK1}. Normal extensions having depth two \cite{KK1} and depth one implying stronger forms of normality such as Hopf-Galois extensions \cite{Ka1}.

In \cite{Y} Young proves there is a relation between the depth of a Hopf algebra $H$ in its smash product with a $H$-module algebra $A$ and the module depth of $A$ in the category of $H$ modules, taking advantage of module depth defined in \cite{Ka2}. As a consequence of this relationship, Kadison and Young proved that the depth of a Hopf algebra $H$ in its Heisenberg double $H\#H^*$ is three. This result and theory is revisited and extended in \cite{HKY}.

It is natural then to ask the following questions: Is there a broader set up in which Young\vtick s result is explained as a consequence of a more general result? Given that for a finite dimensional Hopf algebra $H$, the Drinfel\vtick d and Heisenberg doubles coincide whenever $H$ is cocommutative, can we extend the Kadison and Young\vtick s result on the Heisenberg double to a result concerning the Drinfel\vtick d double?

In this work we give a positive answer to both questions, and factorisation algebras in the sense of \cite{Ma} are the key to solving this problems.

The layout of the paper is as follows: In Section \ref{subalgebra depth} we introduce the reader to definitions of depth as well as some of the most notable results and examples in recent years. Section \ref{FTC} is devoted to depth in finite tensor categories, the concept of module depth which is fundamental to solve our problem is defined here.

In Section \ref{Factorisation Algebras} we introduce factorisation algebras and prove a result for module depth in  Theorem \eqref{prop 1}. We also prove depth inequalities in Theorems \eqref{thm45} and \eqref{h depth ineq}, Inequality \eqref{minimum depth ineq} and Theorem \eqref{augmented algebra equality}. This inequalities mirror previous results in \cite{Ka2} and \cite{HKY}  that are the key components to describe depth of an algebra in a factorisation algebra. 

Section \ref{GSP} is a specialisation of Section \ref{Factorisation Algebras}, we use the machinery already developed to get a connection between the depth of a finite dimensional Hopf algebra extension $R\subseteq H$ with the module depth of certain $H$-module algebra, $Q^{*op}$ in its generalised factorised smash product  $Q^{*op} \subseteq Q^{*op} \#_{\overline{\psi}} H$ in the sense of \cite{Bz1}. In this section the main result is Theorem \eqref{gspd}.

Finaly in Section \ref{dcp} we introduce doble crossed products, as they were first described by Majid in \cite{Ma} and \cite{Ma1} ,which are factorisation algebras and use them to prove a series of results relating module depth and the Drinfel\vtick d double of a finite dimensional Hopf algebra $H$. In particular Theorems \eqref{double crossed product depth}, \eqref{depth of Q and D(H)} and \eqref{Drinfeld h depth} and finally provide a general result concerning the depth of a Hopf algebra $H$ in its Drinfel\vtick d double $D(H)$  in Theorem \eqref{drinfeld minimum depth}.

\subsection{Preliminaries on subalgebra depth}
\label{subalgebra depth}

Throughout this paper all algebras $A$ are associative and with unit over a field $k$ of characteristic zero. We will also consider them to be finite dimensional even though some of the results and definitions provided here are still valid in the infinite dimensional case.  For all ring extensions $B\subseteq A$ we assume $1_B = 1_A$. 

Let $A$ be a ring, we will denote the category of finite dimensional left $A$-modules as $_A\mathcal{M}$. Respectively $\mathcal{M}_A$ is the category of finite dimensional right $A$-modules. $_A\mathcal{M}_A$ is the category of $A$-$A$-bimodules.

Given a ring $A$ and two left $A$-modules $_AV$ and $_AW$ we say they are similar as $A$-modules  (and denote it $V \sim W$) if there are integers $p$ and $q$ such that $V|pW$ and $W|qV$. Where $mN$ means $N\oplus \cdots \oplus N$ $m$ times for any right or left $A$-module $N$. In other words there are split injections of the form $V\hookrightarrow pW$ and $W\hookrightarrow qV$ or equivalently isomorphisms $V\oplus * \cong pW$ and $W\oplus - \cong qV$.

Consider a ring extension $B\subseteq A$. The $n$-th fold tensor power of the regular representation $A_B$ is denoted as $A^{\otimes_B (n)} = A\otimes_B \cdots \otimes_B A$ $n$-times for $n> 0$. Set $A^{\otimes_B (0)} = B$. Notice that $A^{\otimes_B (n)}$ can attain any of the following bimodule structures: $_A{A^{\otimes_B (n) }}_A$, $_B{A^{\otimes_B (n)}}_A$, $_A{A^{\otimes_B (n)}}_B$ and $_B{A^{\otimes_B (n)}}_B$ for $n>0$, each obtained from the preceding by module restriction, and a $B$-$B$-bimodule structure for $n=0$. 

\begin{de}
Let $B\subseteq A$ a ring extension such that $A^{\otimes_B (n)} \sim A^{\otimes_B (n+1)}$ for some integer $n$ as an $X$-$Y$- bimodule with $X,Y \in \{A,B\}$. We say the extension has finite 
\begin{itemize}
\item \textbf{Odd depth}: $d_{odd}(B,A) = 2n + 1$ if $X=Y=B$.
\item \textbf{Even depth}: $d_{ev}(B,A) = 2n$ if $X=B$ and $Y=A$ or $X=A$ and $Y=B$.
\item \textbf{$H$-depth}: $d_h(B,A) = 2n-1$ if $X=Y=A$.
\end{itemize}
\end{de}

The reader will notice that the similarity is compatible with the induction and restriction functors in the tensor category $\mathcal{M}_A$, (symmetrically in $_A\mathcal{M}$), therefore $A^{\otimes_B (n)} \sim A^{\otimes_B (n+1)}$ implies $A^{\otimes_B (n)} \sim A^{\otimes_B (n+r)}$ for all  integers $r > 1$ and for all combinations of $X$-$Y$-bimodules. For this reason we are only interested in the minimum integer $n$ for which the similarity is satisfied, and in any case call it \textbf{minimum depth}. Unless there is a possibility of confusion we denote it $d(B,A)$ for odd and even depth \cite{BDK}, and $d_h(B,A)$ for $H$-depth \cite{Ka4}. In case there is no integer $n$ for which the similarity is satisfied we say the extension is of infinite depth. Furthermore notice that if one of the depths above is finite then all of them are, in fact 
\begin{equation}
\label{finite depth}
d_h(B,A) - 2 \leq d(B,A) \leq d_h(B,A) + 1 
\end{equation}

\begin{ex}
$B\subseteq A$ has minimum depth $1$ if and only if $_BA_B \sim _BB_B$ \cite[27]{BK1}. In this case one can show that thre is an algebra isomorphism $A \cong B\otimes_{Z(B)}A^B$ where $Z(B)$ and $A^B$ denote the center of $B$  and the centraliser of $B$ in $A$ respectively. 

Given a finite group ring extension $kH\subseteq kG$ depth $2$ is equivalent to normality of the group pair $H\vartriangleleft G$. A similar result is true for Hopf algebra extensions of finite index $R\subseteq H$ \cite{BK1}. Ring extensions of depth 2 are bialgebroid-Galois extensions (given a generator condition on $A_B$ is met). The case $H$ depth equals $1$ is the $H$-separability condition \cite{Ka4} since as $A$-$A$ bimodules the extension satisfies $A\otimes_B A | qA$ for some integer $q$. Hence $A$ is a separable extension of $B$. Hopf Galois extensions have minimum depth $\leq2$ as well \cite{Ka1}.
\end{ex}

\begin{ex}
Let $B\subseteq A$ be an extension of semisimple complex matrix algebras. Let $M:K_0(B)\longrightarrow K_0(A)$ be the $r\times s$ induction matrix where $r$ and $s$ are the number of irreducible representations of $B$ in $K_0(B)$ and of $A$ in $K_0(A)$ respectively.
Let $S$ be the order $r$ symmetric matrix defined by $S := MM^T$. In \cite{BKK} it is shown that the odd depth satisfies $d_{odd}(B,A) = 2n+1$ if $S^{(n)}$ and $S^{(n+1)}$ have an equal number of zero entries. The minimum even depth can be computed in the same manner by looking at the zero entries of the powers of $S^{(m)}M$. The overall minimum depth is given by the least $n \in \mathbb{N}$ such that
\begin{equation}
M^{(n+1)} \leq qM^{(n-1)}
\end{equation}
for some natural number $q$, each $(i,j)$-entry. $M^{(0)} := I_r$, $M^{(2n)} = S^{(n)}$, and $M^{(2n+1)} = S^{(n)}M$.
Let $N$ be the order $s$ symmetric matrix defined by $N:= M^TM$. The minimum $H$-depth of the extension $B\subseteq A$ is given by the least natural number $n$ such that the zero entries of $N^{(n)}$ stabilise \cite{Ka7}. From matrix definitions it follows that a subalgebra pair of semisimple algebras $B\subseteq A$ over a field of characteristic zero is always finite depth. In characteristic $p$ finite depth holds if either $B$ or $A$ is a separable algebra\cite[Corollary 2.2]{KY} .
\end{ex}

In \cite{BDK} the authors provide formulas for the depth of the extension of group algebras of the symmetric  (respectively the alternating) group on $n$ and $n+1$ letters over the complex numbers, $d(\mathbb{C}\mathcal{S}_n, \mathbb{C}\mathcal{S}_{n+1}) = 2n+1$, (respectively $d(\mathbb{C}\mathcal{A}_n,\mathbb{C}\mathcal{A}_n) \leq 2(n - \lceil \sqrt {n}\rceil) + 1$). Moreover the following string of inequalities relating the different depths of finite group algebra extensions over a variety of rings holds \cite{BDK}:

\begin{equation}
\label{string group depth}
d_0(H,G) \leq d_p(H,G) \leq d_R(H,G) \leq d_{\mathbb{Z}}(H,G) \leq d_c(H,G)
\end{equation}

In \cite{SD} a formula for the twisted complex extension of algebras over the symmetric groups of order $n$ and $n+1$ is given: $d(\mathbb{C}_\alpha \mathbb{S}_n, \mathbb{C}_\alpha \mathbb{S}_{n+1}) = 2(n - \lceil\frac{\sqrt{8n+1}-1}{2}\rceil)+1$, where $\alpha$ is a $2$-cocycle. In \cite{HKS} by exploiting the structure of corings and entwining structures between a coalgebra $C$ and an algebra $A$; $(A,C)_\psi$, as well as extensions of Galois-corings,  the authors provide a framework to explain this result via extensions of crossed product corings of the form $D\#_\sigma H \subseteq D\#_\sigma G$ where $H<G$ is a finite group extension, $D$ is a twisted module algebra and $\alpha$ a $2$-cocycle and extend the string of inequalities \eqref{string group depth} to

\begin{equation*}
d(D\#_\sigma H, D\#_\sigma G) \leq d_0(H,G) \leq d_p(H,G) \leq 
\end{equation*}
\begin{equation}
\label{string group depth 2}
d_R(H,G) \leq d_{\mathbb{Z}}(H,G) \leq d_c(H,G)
\end{equation}

Other interesting results in depth theory such as depth of Young subgroups of the symmetric group of order $n$, depth of subgroups of Suzuki groups, depth of the maximal subgroups of Ree groups, the depth of subgroups of $PSL(2,q)$, and the depth of the Taft algebra of order $n^2$ in its Drinfel\vtick d double, can be found in \cite{FKR},\cite{HHP},\cite{HHP2},\cite{F},\cite{HKL}, respectively. 

\section{Depth in finite tensor categories and the quotient module of a finite dimensional extension of Hopf algebras}
\label{FTC}

Let $H$ be a Hopf algebra with coproduct denoted $\Delta\ (h) = h_1\otimes h_2$ (where the summation over the index $(h)$ is understood) and antipode $S$. Let $R$ be a Hopf subalgebra of $H$, so $\Delta (R) \subseteq R\otimes R$ and $S(R) = R$. Since we are only considering finite dimensional right or left $H$-modules, it is important to note that the category $_H\mathcal{M}$ (respectively $\mathcal{M}_H$) is a finite tensor category, in particular is a \textit{Krull-Schmidt} category since it is monoidal and all objects are of finite length, hence satisfy the Krull-Schmidt decomposition theorem. In \cite{Ka2} it was shown that the tensor powers of $H$ over $R$  reduce to tensor powers of the quotient module $Q= H/R^+H$ of the extension $R\subseteq H$ .  This is $H^{\otimes_R (n+1)} \cong H\otimes Q^{\otimes (n)}$ given by:
\begin{equation}
\label{depth iso}
x^1\otimes_ R x^2 \otimes_R \cdots \otimes_R x^{n+1} \longmapsto x^1 x^2_1 \cdots x^{n+1}_1\otimes \overline{x^2_2\cdots x^{n+1}_2}\otimes \cdots \otimes \overline{x^{n+1}_n} 
\end{equation}

Given this setting is worth noting that depth of the extension $R\subseteq H$ is related to the module depth of $Q$ in the following sense. Let $W$ be a left $H$ module, denote its $n$-th truncated algebra $T_n(W) = W\oplus W^{\otimes (2)}\oplus \cdots W^{\otimes (n)}$ for $n\geq 1$ and $T_0(W) =$ $_Hk$:

\begin{de}
Let $W$ be a left (or right) module over a Hopf algebra $H$. We say $W$ has module depth $n > 1$ and denote it $d(W,_H\mathcal{M}) = n$ if  $T_n(W) \sim T_{n+1}(W)$ and depth $0$ if $W$ is a copy of direct sums of the trivial module $_Hk$.
\end{de}

As before, and for the same reasons, note that depth $n$ implies depth $n+r$ for $r\geq 1$. Again we are  only interested in minimum module depth when it occurs. If the module $W$ has finite module depth we say it is algebraic. \\

Recall that the representation ring or Green ring $A(H)$ of a Hopf algebra $H$ is generated by the set of isocalsses of indecomposable $H$-modules, with the sum and product of isoclasses given by the isoclasses of the direct sum and tensor product respectively. An algebraic object in the representation ring of a finite group algebra is an object satisfying a polynomial with integer coefficients. This notion appears for example in \cite{Fe} and can naturally be extended to that of an algebraic object in the Green ring $A(H)$ of a Hopf algebra $H$. The notions of finite depth in $_H\mathcal{M}$ and being an algebraic element in $A(H)$ are quivalent, a short proof of this can be found in \cite{H}.\\

Note that for all  $W \in$ $_H\mathcal{M}$ and for all $m>0$ one has $W^{\otimes (m)}|qW^{\otimes (m+1)}$ for some integer $q$. Hence for all $r\geq 1$ all indecomposable constituents of $W^{\otimes (m)}$ are contained in the set of indecomposable constituents of $W^{\otimes (m+r)}$. In particular $Ind(W^{\otimes(m)}) \subseteq Ind(W^{\otimes (m+1)})$. For this reason $d(W,_H\mathcal{M}) = n$ implies $W^{\otimes (n)}\sim W^{\otimes (n+1)}$. A somewhat converse of this is the following:

\begin{pr}
\label{module coalgebra depth}
Let $W$ be a left $H$-module coalgebra (a coalgebra in $_H\mathcal{M}$), then $W^{\otimes (n)} \sim W^{\otimes (n+1)}$ implies $d(W,_H\mathcal{M}) \leq n$.
\end{pr}

\begin{proof}
Just recall that as a coalgebra map in $_H\mathcal{M}$ the coproduct $\Delta$ splits via the counit $\varepsilon$. Then $Ind(W^{\otimes (n)}) \subseteq Ind(T_n(W)$. Hence the result.
\end{proof}

Notice that given an extension of Hopf algebras $R\subseteq H$ their quotient module $Q=H/R^+H$ is a left $H$-module coalgebra via $g\cdot\Delta(\overline{h}) = \overline{g_1h_1} \otimes \overline{g_2h_2}$. Then we have the following:

\begin{pr}\cite{Ka2}
\label{hdepth}
Let $R\subseteq H$ be a Hopf algebra extension, then 
\begin{equation}
d_h(R,H) = 2d(Q,_H\mathcal{M})+1
\end{equation}
\end{pr}

\begin{proof}
This is a consequence of applying equation \eqref{depth iso} on both sides of $H^{\otimes_R (n)}\sim H^{\otimes_R (n+1)}$ to obtain $H\otimes Q^{\otimes (n)} \sim H\otimes Q^{\otimes (n+1)}$ and then applying $k\otimes_H -$ on the left on both sides of the similarity  to obtain $Q^{\otimes (n)} \sim Q^{\otimes (n+1)}$. Then apply proposition \eqref{module coalgebra depth} to obtain $d_h(R,H) \geq 2d(Q,_H\mathcal{M})+1$. The other inequality follows from the definitions. 
\end{proof}

\begin{ex}
Let $R\subseteq H$, following \cite{HKS} one verifies that $H\otimes Q$ is a Galois coring, a coalgebra in $_H\mathcal{M}_H$ such that $H\otimes_R H \cong H\otimes Q$ via $a\otimes b \longmapsto agb$ where $g$ is the grouplike element $1\otimes\overline{1}$ in $H\otimes Q$ and multiplication is given via an entwining $\psi: Q\otimes H \longrightarrow H\otimes Q$. \cite[Proposition 4.2]{HKS} relates the module depth of $H\otimes Q$ as an $H$-$H$-bimodule and the module depth of $Q$ as an $H$ module via
\begin{equation}
d(H\otimes Q, _H\mathcal{M}_H) = 2d(Q,\mathcal{M}_H) +1
\end{equation}
\end{ex}

\section{Factorisation algebras}

\subsection{Factorisation algebras and depth}
\label{Factorisation Algebras}

\begin{de}
Let $A$ and $B$ be to finite dimensional algebras, $m_A$ and $m_B$ their respective multiplications and let: 
\begin{equation}
\psi : B\otimes A \longrightarrow A\otimes B \quad ; \quad b\otimes a \longmapsto a_{\alpha}\otimes b^{\alpha}  
\end{equation}
such that 
\begin{equation}
\psi (1_B\otimes a) = a\otimes 1_B,\quad \psi (b\otimes 1_A) = 1_A\otimes b
\end{equation}
for all $a \in A$ and $b\in B$. Moreover suppose $\psi$ satisfies the following octagon:
\begin{equation*}
(A\otimes m_B)\circ (\psi \otimes B)\circ (B\otimes m_A \otimes B)\circ (B\otimes A \otimes \psi) =
\end{equation*}
\begin{equation*}
(m_a\otimes B)\circ (A\otimes \psi)\circ (A\otimes m_B \otimes A) \circ (\psi \otimes B \otimes A)
\end{equation*}
this means that for all $a,d \in A$, and all $b,c \in B$
\begin {equation}
\label{factorisation associativity}
(ad_\alpha)_\beta \otimes b^\beta c^\alpha = a_\beta d_\alpha \otimes (b^\beta c)^\alpha
\end{equation}
We call $\psi$ a factorisation of $A$ and $B$ and $A\otimes_\psi B$ a factorisation algebra of $A$ and $B$. It is a  unital associative algebra with product
\begin{equation}
(a\otimes b) (c\otimes d) = a\psi (b\otimes c)d = ac_\alpha \otimes b^\alpha d
\end{equation}
where $a,c \in A$, $b,d \in B$ and the unit element is $1_A \otimes 1_B$.  Besides $A$ and $B$ are $A\otimes_\psi B$ subalgebras via the inclusion. 
\end{de}

\begin{ex}
Of course setting $\psi(b\otimes a) = a\otimes b$ yields the algebra $A\otimes B$. On a more sophisticated manner, let $H$ be a Hopf algebra and $A$ a left $H$-module algebra. Suppose $H$ measures $A$: $h\cdot (ab) = (h_{1}\cdot a )(h_2 \cdot b)$, $h\cdot 1_A = \varepsilon(h) 1_A$ for all $h \in H$ and $a,b \in A$. define $\psi$ in the following way: 
\begin{equation*}
\psi: H\otimes A; \quad h\otimes a\longmapsto h_1\cdot a \otimes h_2
\end{equation*}
then the product becomes 
\begin{equation*}
(a\otimes h)(b\otimes g) = a\psi(h\otimes b)g = a(h_1 \cdot b \otimes h_2)g = ah_1\cdot b\otimes h_2 g 
\end{equation*}
It is a routine exercise to verify that $A\otimes_\psi H$ is a factorisation algebra. It is also immediate that such factorisation algebra is the smash product of $A$ and $H$, $A\# H$.
\end{ex}

Now let $A\otimes_\psi B$ be a factorisation algebra via $\psi: B\otimes A \longmapsto A\otimes B$. For the sake of brevity we will denote it $S_\psi = A\otimes_\psi B$. We point out that due to multiplication in $S_\psi$ and the fact that both $A$ and $B$ are subalgebras of $S_\psi$ we get that the $n$-th tensor power of the regular representation of the extension $B\subseteq S_\psi$ is an $S_\psi$-bimodule via:
\begin{equation*}
(a\otimes_\psi b)(a_1 \otimes b_1 \otimes_B \cdots \otimes_B a_n\otimes b_n)(c \otimes_\psi d) =
\end{equation*}  
\begin{equation}
\label{psi linearity}
= aa_{1\alpha} \otimes b^\alpha b_1\otimes_B \cdots \otimes_B a_nc_\alpha\otimes b_n^\alpha d
\end{equation}

\begin{thm}
\label{prop 1}
Let $A$ and $B$ be algebras, $\psi: B\otimes A \longmapsto A\otimes B$ a factorisation and $S_\psi$ the corresponding factorisation algebra. Then
\begin{equation}
S_\psi ^{\otimes_B (n)} \cong A^{\otimes(n)} \otimes B
\end{equation}
as $S_\psi$-bimodules via
\begin{equation*}
\theta_n (a_1\otimes b_1 \otimes_B \cdots \otimes_B a_n\otimes b_n) =
\end{equation*}
\begin{equation} 
a_1\otimes a_{2_{\alpha 1}}\otimes \cdots \otimes a_{n_{\alpha (n-1)}}\otimes b_1^{\alpha 1}b_2^{\alpha 2}\cdots b_n
\end{equation}
with inverse 
\begin{equation}
\theta_n ^{-1} (a_1\otimes \cdots \otimes a_n\otimes b) = a_1\otimes 1_B \otimes_B \cdots \otimes_B a_{n-1}\otimes 1_B \otimes_B a_n \otimes b 
\end{equation}
\end{thm}

\begin{proof}

First of all $A$-$B$ linearity is given by multiplication either in $A$ or $B$. Then notice that using multiplication in $S_\psi$ and equation \eqref{factorisation associativity} one gets $\theta_n$ in the following way:

\begin{equation*}
a_1 \otimes b_1 \otimes_B a_2 \otimes b_2 \otimes _B a_3 \otimes b_3 \otimes_B \cdots \otimes_B a_n \otimes b_n \longmapsto \end{equation*}
\begin{equation*}
a_1 \otimes (1_A \otimes b_1)(a_2 \otimes b_2) \otimes _B a_3 \otimes b_3 \otimes_B \cdots \otimes_B a_n \otimes b_n =
\end{equation*}
\begin{equation*}
a_1 \otimes  a_{2(\alpha 1)} \otimes b_1^{\alpha 1}b_2 \otimes _B a_3 \otimes b_3 \otimes_B \cdots \otimes_B a_n \otimes b_n \longmapsto 
\end{equation*}
\begin{equation*}
 a_1 \otimes  a_{2(\alpha 1)} \otimes (1_A \otimes b_1^{\alpha 1}b_2) (a_3 \otimes b_3) \otimes_B \cdots \otimes_B a_n \otimes b_n = 
\end{equation*}
\begin{equation*}
a_1 \otimes  a_{2(\alpha 1)} \otimes a_{3(\alpha 2)} \otimes (b_1^{\alpha 1}b_2)^{\alpha 2}b_3 \otimes_B \cdots \otimes_B a_n \otimes b_n =
\end{equation*}
\begin{equation*}
a_1 \otimes  a_{2(\alpha 1)} \otimes a_{3(\alpha 2)} \otimes b_1^{\alpha 1}b_2^{\alpha 2}b_3 \otimes_B \cdots \otimes_B a_n \otimes b_n 
\end{equation*}
Repeat this process to the right $n-1$ times to get 
\begin{equation*}
a_1\otimes a_{2(\alpha 1)} \otimes \cdots a_{n(\alpha n-1)} \otimes b_1^{\alpha 1}b_2^{\alpha 2}\cdots b_n
\end{equation*}

Now it is easy to check that:

\begin{equation*}
\theta_n^{-1}\circ\theta_n (a_1\otimes b_1 \otimes_B \cdots \otimes_B a_n \otimes b_n) = a_1\otimes b_1 \otimes_B \cdots \otimes_B a_n \otimes b_n
\end{equation*}
Indeed, notice that by multiplication and equation \eqref{factorisation associativity}  for the last tensor powers in $\theta_n (a_1\otimes b_1 \otimes_B \cdots \otimes_B a_n \otimes b_n)$ one has:
\begin{equation*}
\cdots 1_B \otimes_B a_{n(\alpha (n-1))}\otimes b_1^{\alpha 1}b_2^{\alpha 2} \cdots b_n  =  \cdots 1_B \otimes_B (1_A \otimes b_1^{\alpha 1}\cdots b_{n-1})(a_{n}\otimes b_n) 
\end{equation*}
but $ 1_A \otimes b_1^{\alpha 1}\cdots b_{n-1} \in B$ and moves along to the left over the $B$-tensor, so we get $= \cdots \otimes b_1^{\alpha 1}\cdots b_{n-1}\otimes (a_{n}\otimes b_n)$ Repeat this process to the left $n-1$ times and we get what we want, that is $\theta_n^{-1}\circ \theta_n = id_{S_\psi ^{\otimes_B (n)}}$. Checking that $\theta\circ\theta^{-1} = Id_{A\otimes B^{\otimes n}}$ is straightforward. 

Notice that the $S_\psi$-linearity of $\theta$ as well as of $\theta^{-1}$ is given by equation \eqref{psi linearity} together with equation \eqref{factorisation associativity}. 
\end{proof}

Since $A$ and $B$ are factor algebras in $S_\psi$ one identifies them with $A\otimes_\psi 1_B$ and $1_A\otimes_\psi B$ respectively, hence one can think of their regular representations as objects in $_{S_\psi}\mathcal{M}_{S_\psi}$.

\begin{thm}
\label{thm45}
Let $A\otimes_\psi B$ be a factorisation algebra with $_{S_\psi}\mathcal{M}$ a Krull-Schmidt category. Then
\begin{equation}
\label{factorisation h depth}
d_h(B,S_\psi) \leq 2d(A, _{S_\psi}\mathcal{M}_{S_\psi}) + 1
\end{equation}
\end{thm}

\begin{proof}
Let $d(A,_{S_\psi}\mathcal{M}_{S_\psi}) = n$. Since $_{S_\psi}\mathcal{M}_{S_\psi}$ is a Krull-Schmidt category, standard face and degeneracy functors imply $A^{\otimes_B (m)} | A^{\otimes_B (m+1)}$ for $m\geq 0$. Then $T_n(A) \sim T_{n+1}(A)$ implies $A^{\otimes (n+1)} \sim A^{\otimes (n)}$. Tensoring on the right by $(- \otimes B)$ one gets $A^{\otimes (n+1)}\otimes B \sim A^{\otimes (n)}\otimes B$. By Theorem \eqref{prop 1} this is equivalent to $(A\otimes_\psi B)^{\otimes_B (n+1)} \sim (A\otimes_\psi B)^{\otimes_B (n)}$. This by definition is $d_h(B,S_\psi) \leq 2n + 1$.
\end{proof}

\begin{co}
\label{h depth ineq}
Assume $A \in _B\mathcal{M}_B$. The odd and even depth of $B\subseteq S_\psi$ satisfy the following:
\begin{equation}
2d(A,_{S_\psi}\mathcal{M}_{S_\psi})+2 \leq d(B,S_\psi) \leq 2d(A,_{S_\psi}\mathcal{M}_{S_\psi})+3
\end{equation}
\end{co}

\begin{proof}
Assume $d(A,_{S_\psi}\mathcal{M}_{S_\psi}) = n$ then applying restriction of modules, (the forgetful functors $_{S_\psi}\mathcal{M}_{S_\psi} \longrightarrow _B\mathcal{M}_{S_\psi} \longrightarrow _B\mathcal{M}_B$), and using the definition for odd and even depth found in the introduction yield the result.
\end{proof}

Given $A$ is a left $B$-module algebra, there is a series of other inequalities one gets from similarities of tensor powers of $A$ in $_B\mathcal{M}_{S_\psi}$ and in $_B\mathcal{M}_B$, in particular by applying the techniques of Theorem \eqref{prop 1} in the afore mentioned cases one gets:
\begin{equation}
\label{minimum depth ineq}
2d(A,_B\mathcal{M}_B)+1 \leq d(B,S_\psi) \leq 2d(A,_B\mathcal{M}_{S_\psi})+2
\end{equation}

\begin{co}
\label{augmented algebra equality}
Let $B$ be an augmented algebra. Then the inequality \eqref{factorisation h depth} becomes an equality. 
\end{co}

\begin{proof}
It suffices to tensor on the right by $ (- \otimes_B k) $ on both sides of the similarity $A^{\otimes (n+1)} \otimes B \sim A^{\otimes (n)}\otimes  B$ to get $A^{\otimes n+1} \sim A^{\otimes n}$.
\end{proof}

\begin{co}
\label{smash depth}
Let $H$ be a finite dimensional Hopf algebra and $A$ a left $H$-module algebra measured by $H$. Consider $A\# H$, the smash product of $A$ and $H$. Then 
\begin{equation}
d_h(H,A\# H) = 2d(A, _{A\#H}\mathcal{M}_{A\# H}) + 1
\end{equation}
\end{co}

\begin{ex}\cite[Example 6,6]{HKY}
\label{Heisenberg algebra}
Considering $H$ and $H^*$ a Hopf algebra and its dual and forming the Heisenberg double $H \# H^*$  via $f\rightharpoonup h = h_1f(h_2)$ one verifies 
\begin{equation}
d(H^*, H\#H^*) = 3
\end{equation}
\end{ex}

\begin{proof}
Restricting to $_H\mathcal{M}_H$ one recovers the formula for odd depth $d(H,H^*\# H) = 2d(H^*, _H\mathcal{M}_H)+1$ on the left hand side of equation \eqref{minimum depth ineq}. Considering $H$ is a Frobenius algebra, by self duality $_{H^*}H^*_{H^*} \cong$ $ _HH^*_H$. Then, since $d(H^*, _{H^*}\mathcal{M}_{H^*}) = 1$ the result follows.
\end{proof}

\subsection{Minimum depth and the generalised factorised smash product}
\label{GSP}

As an example of the previous results we will use a construction by \textit{Brzezinski} \cite[Section 3]{Bz1} to get a result on depth of a finite dimensional Hopf algebra extension via the generalised factorised smash product of $Q^{*op}$ and $H$.

Let $C$ be a $k$-coalgebra, recall that its $k$-dual $C^*$ is a $k$-algebra via the convolution product and evaluation: $C\otimes C^* \longrightarrow k$; $c\otimes \theta \longmapsto \theta(c)$.

\begin{de}

Let  $A$ be an algebra and $C$ a coalgebra. Let $(A,C)_\psi$ be an entwining structure as defined in \cite{BZ}. We say $(A,C)_\psi$ is factorisable \cite[Section 3]{Bz1}if there exists a unique 

\begin{equation}
\label{eq}
\overline{\psi}:A\otimes C^*\longrightarrow C^*\otimes A
\end{equation}
such that the following diagram commutes:
\begin{equation}
\label{eq24}
\begin{tikzcd}
C\otimes A \otimes C^*  \arrow{r}{C\otimes \overline{\psi} }  \arrow{d}[swap]{\psi\otimes\ C^*} & C\otimes C^*\otimes A  \arrow{d}{ev_{C^*}\otimes A} \\
A\otimes C\otimes C^*  \arrow{r}[swap]{A\otimes ev_{C^*}}                  &  A 
\end{tikzcd}
\end{equation}
We write $\overline{\psi}(a\otimes \theta) = \theta_i \otimes a^i$ for every $a \in A$ and $\theta \in C^*$, summation over $i$ is understood.

\end{de}

Furthermore entwining structures $(A,C)_\psi$ are factorisable provided $C$ is $k$ projective \cite{Bz1} when $k$ is a general commutative ring.

Let now $R\subseteq H$ be a finite dimensional Hopf algebra extension, let $Q$ be their quotient module coalgebra.

Let $C = Q$. We have that $(H,Q)_\psi$ defined by $\psi(\overline{h}\otimes g) = g_{(1)}\otimes \overline{hg_{(2)}}$  for all $h , g\in H$ is an entwining structure. Since $Q$ is a vector space over $k$ it is $k$-projective. Now consider
\begin{equation*}
\overline{\psi}:H\otimes Q^* \longrightarrow Q^*\otimes H
\end{equation*} 
\begin{equation}
\label{eq25}
h\otimes \theta\longmapsto (h_{(2)}\rightharpoonup \theta)\otimes h_{(1)}
\end{equation}
here the action is
\begin{equation*}
(h\rightharpoonup \theta)(-) = \theta (\overline{-h})
\end{equation*}
For $h, - \in H$.

Then the diagram \eqref{eq24} is satisfied and $(H,Q)_\psi$ is factorisable via $\overline{\psi}$ as defined above.

\begin{de}
Consider now the algebra $Q^{*op}$, that is to say, for all $\theta, \gamma \in Q^{*op}$ and $h\in H$ we have $<\overline{h},\theta\gamma> = <\overline{h_{(2)}},\theta><\overline{h_{(1)}},\gamma>$. Then we form the generalised factorised smash product algebra $Q^{*op}\#_{\overline{\psi}} H$ with product given by 
\begin{equation}
(\theta\# h)(\gamma\# g) = \theta(h_{(2)}\rightharpoonup \gamma)\# h_{(1)}g
\end{equation}
\end{de}

Of course we identify $Q^{*op}$ with $Q^{*op}\# 1_H$ and $H$ with $\varepsilon_H\# H$. Let $h, g \in H$ and $\theta, \gamma \in Q^{*op}$, the multiplication yields
\begin{equation}
(\varepsilon_H\# h)(\varepsilon_H\# g) = \varepsilon_H(h_{(2)}\rightharpoonup\varepsilon_H)\#h_{(1)}g = \varepsilon_H \# hg
\end{equation}
and 
\begin{equation}
(\theta\# 1_H)(\gamma\# 1_H) = \theta(1_H\rightharpoonup \gamma)\# 1_H1_H = \theta\gamma\# 1_H 
\end{equation}

Hence both $H$ and $Q^{*op}$ are subalgebras in $Q^{*op}\#_{\overline{\psi}} H$. Moreover $Q^{*op}\#_{\overline{\psi}} H$ is an $H$ and $Q^{*op}$ left and right module via multiplication.

\begin{thm}
\label{gspd}
Let $R\subseteq H$ be a finite dimensional Hopf algebra extension and $Q = H/R^+ H$ its quotient module coalgebra. Let $Q^{*op}\#_{\overline{\psi}} H$ be the generalised factorised smash product. Then 
\begin{equation}
d_{odd}(H,Q^{*op}\#_{\overline{\psi}} H) = d_h(R,H)
\end{equation}
\end{thm}

\begin{proof}
Assume then that $d_{odd}(H, Q^{*op} \#_{\overline{\psi}} H) = 2n + 1$. Apply corollary \eqref{smash depth} to get $d(Q^{* op}, _H\mathcal{M}) = n$. Since $Q^* \longrightarrow H^*$ is a Frobenius extension \cite{Ka2} one can see that this implies by self duality and opposing categories $d(Q,\mathcal{M}_H) = n$. By theorem \eqref{hdepth} we get $d_h(R,H) = 2n + 1$ which is what we wanted.
\end{proof}

\section{Depth of double crossed products and the Drinfeld double}
\label{dcp}

Double crossed products were introduced by Majid in \cite{Ma1} with the purpose of dealing with factorisations of Hopf algebras and generalising the construction of the Drinfel\vtick d double of a Hopf algebra. In the subsequent literature one finds that such product is referred to as bicrossed product in an attempt to reconcile the term with its roots in group theory. We will refer to it simply as double crossed product since it makes no difference for our purposes. For a comprehensive account of the theory treated in this section  the reader may refer to \cite{Ma1} and \cite{Ma} as well as \cite{Bu2}, \cite{ABM} and \cite{Ta}.

Let $H , K$ be two Hopf algebras such that $H$ is a right $K$-module coalgebra and $K$ a left $H$-module coalgebra: 

\begin{de}
\label{matched pair}
We say $K$ and $H$ are a matched pair if there are coalgebra maps
\begin{equation}
\alpha: H\otimes K \longrightarrow H; \quad h\otimes k \longmapsto h\triangleleft k
\end{equation}
\begin{equation}
\beta : H\otimes K \longrightarrow K; \quad h\otimes k \longmapsto h\triangleright k
\end{equation}
such that the following compatibility conditions hold:
\begin{equation}
(hg)\triangleleft k = \sum (h\triangleleft(g_1\triangleright k_1))(g_2\triangleleft k_2) ; \quad 1_H\triangleleft k = \varepsilon_K(k) 1_H
\end{equation}
\begin{equation}
h\triangleright (kl) = \sum (h_1\triangleright k_1)((h_2\triangleleft k_2)\triangleright l) ; \quad h \triangleright 1_K = \varepsilon_H(h) 1_K
\end{equation}
\end{de}

\begin{de}
\label{bicrossed product}
Let $H,K$ be a matched pair of Hopf algebras. The $\mathbf{double}$ $\mathbf{crossed}$ $\mathbf{product}$ of $H$ and $K$ is a Hopf algebra denoted $K\bowtie H$ over the set $K\otimes H$ with product given by
\begin{equation}
(k\bowtie h)(l\bowtie g) = \sum k(h_1\triangleright l_1)\bowtie (h_2\triangleleft l_2)g
\end{equation}
Moreover, the coalgebra structure and antipode are defined in the following way:
\begin{equation}
\Delta(k\bowtie h) = k_1\bowtie h_1 \otimes k_2\bowtie h_2
\end{equation}
\begin{equation}
\varepsilon (k\otimes h) = \varepsilon_K(k)\varepsilon_H(h)
\end{equation}
\begin{equation*}
S(k\bowtie h) = (1_K\bowtie S_H(h))(S_K(k)\bowtie k) 
\end{equation*}
\begin{equation}
= S_H(h_1)\triangleright S_K(k_1) \bowtie S_H(h_2)\triangleleft S_K(k_2)
\end{equation}
\end{de}

\begin{ex}\cite{Bu2}
Suppose $H$ is a Hopf algebra and $L$ and $A$ two sub-Hopf algebras, such that $H \cong A\otimes_\psi L$ is a factorisation, then the multiplication $m: L \otimes A \longrightarrow H$ defined by $a\otimes l \longmapsto al$ is a bijection. This implies $A\bigcap L = k$. Then following \cite{Ma} one concludes that $H$ is a double crossed product in the following way: Consider the map
\begin{equation*}
\mu: L\otimes A \longrightarrow A\otimes L; \quad l\otimes a \longmapsto m^{-1}(la)
\end{equation*}
then define 
\begin{equation*}
\triangleright: L\otimes A\longrightarrow A; \quad l\triangleright a = ((Id\otimes \varepsilon_L)\circ \mu)(l\otimes a)
\end{equation*}
\begin{equation*}
\triangleleft: L\otimes A\longrightarrow L; \quad l\triangleleft a = ((\varepsilon_A \otimes Id)\circ \mu)(l\otimes a)
\end{equation*}
\end{ex}

\begin{ex}\cite{ABM}
Let $G$ and $K$ be groups, $A=k[G]$ and $H=k[K]$ the corresponding group algebras. There is a bijection between the set of all matched pairs Hopf algebras $(k[G], k[K], \triangleleft , \triangleright)$ and the set of all matched pairs of groups $(G,K, \overline{\triangleleft} , \overline{\triangleright})$, in the sense of Takeuchi \cite{Ta}. The bijection is given such that there is an isomorphism of Hopf algebras $k[G] \bowtie k[K] \cong k[G\bowtie K]$, where $G\bowtie K $ is the Takeuchi double crossed product of groups, \cite{Kas}. 
\end{ex}

\begin{pr}
Double crossed products are factorisation algebras
\end{pr}

\begin{proof}
Let $K\bowtie H$ be a double crossed product, define $\psi: H\longrightarrow K$ in the following way:
\begin{equation}
\psi(h\otimes k) = (h_1\triangleright k_1)\bowtie (h_2\triangleleft k_2)
\end{equation}
Clearly $\psi$ satisfies the axioms of a factorisation and in this manner the double crossed product of $K$ and $H$ is a factorisation algbera
\end{proof}

\begin{thm}
\label{double crossed product depth}
The $H$ depth of the subalgebra $H$ in a double crossed product $K\bowtie H$ is given by
\begin{equation}
d_h(H,K\bowtie H) = 2d(K,_{K\bowtie H}\mathcal{M}_{K\bowtie H}) + 1
\end{equation}
\end{thm}

\begin{proof}
This is just Corollary \eqref{augmented algebra equality}
\end{proof}

Consider the double crossed  product of two Hopf algebras $H$ and $K$, we know already that $K\bowtie H$ is itself a Hopf algebra, and both $H$ and $K$ Hopf subalgebras via inclusion and multiplication. Then we have two cases for the quotinet module of the pair of extensions: $Q_K = K\bowtie H / K^+ K\bowtie H$ and $Q_H = K\bowtie H / H^+ K\bowtie H$. Of course both quotient modules are objects in $_{K\bowtie H}\mathcal{M}_{K\bowtie H}$ and both restrict naturally as bimodules over their respective subalgebra.

\begin{thm}
\label{depth of Q and D(H)}
Let $K\bowtie H$ a double crossed product of Hopf algebras, such that $d(H, K\bowtie H)$ is finite. Then
\begin{equation}
d(Q_H, _{K\bowtie H}\mathcal{M}_{K\bowtie H}) = d(K, _{K\bowtie H}\mathcal{M}_{K\bowtie H})
\end{equation}
\end{thm} 
  
\begin{proof}
This results from combining Theorem \eqref{double crossed product depth} and Theorem \eqref{hdepth}. 
\end{proof}

This result should not come as a surprise due to the normal basis property for a finite dimensional Hopf algebra extension $R\subseteq H$ \cite{Mo}.

Consider now the case where $H$ is a Hopf algebra and $K = H^{*cop}$ the $H$ dual with cooposite coalgebra structure. Denote $S$ and $S^*$ the respective antipodes and $\overline{S}$ and $\overline{S^*}$ their composition inverse. Let $H$ act on $H^*$ by the left coadjoint action:
\begin{equation}
h\rrhu f = h_1\rightharpoonup f \leftharpoonup \overline{S}(h_2)
\end{equation}
Similarly $H^*$ acts on $H$ from the right by the transpose of the right coadjoint action :
\begin{equation}
h\llhu f = \overline{S^*}(f_1)\rightharpoonup h \leftharpoonup f_2 
\end{equation}

\begin{de}
Given a Hopf algebra $H$ define its $\mathbf{Drinfel\vtick d}$ $\mathbf{double}$ $D(H)$ as the double crossed product of $H$ and $H^{*cop}$ where the product is given by
\begin{equation}
(f\bowtie h)(l \bowtie g) = f(h_1\rrhu l_2)\bowtie(h_2 \llhu l_1)g
\end{equation}
\end{de}

\begin{thm}
\label{Drinfeld h depth}
Let $H$ be a Hopf algebra, then 
\begin{equation}
d_h(H,D(H)) = 2d(H^{*cop}, _{D(H)}\mathcal{M}_{D(H)}) + 1
\end{equation}
\end{thm}

\begin{proof}
This is a just Theorem \eqref{double crossed product depth}. 
\end{proof}

\begin{pr}
\label{drinfeld minimum depth}
Let $H$ be Hopf algebra, then
\begin{equation}
3 \leq d(H,D(H))
\end{equation}
\end{pr}

\begin{proof}
Combine Equation \eqref{minimum depth ineq} and Theorem \eqref{depth of Q and D(H)} and follow the reasoning in Example \eqref{Heisenberg algebra} to get the result.
\end{proof}

This in particular means that the extension $H\subseteq D(H)$ is never normal.

\end{document}